\documentclass[preprint,12pt]{elsarticle}

\usepackage{amscd}
\usepackage{amsmath}
\usepackage{amsthm}
\usepackage{amsfonts}
\usepackage{graphicx}
\usepackage{amssymb}
\usepackage{color}
\usepackage{amsbsy}
\usepackage{graphicx}
\usepackage{amssymb,color,amsbsy}
\usepackage{mathtools}
\usepackage{leftindex}

\usepackage{float}
\usepackage{comment}

\usepackage{stmaryrd}

\usepackage[ruled]{algorithm2e}

\usepackage[matrix,arrow]{xy}

\DeclareMathAlphabet{\mathpzc}{OT1}{pzc}{m}{it}

\usepackage{pgfpages}
\usepackage{tikz}
\usetikzlibrary{shapes.geometric}
\usetikzlibrary{decorations.pathmorphing}
\usetikzlibrary{arrows}
\usetikzlibrary{backgrounds}
\usetikzlibrary{positioning}
\usetikzlibrary{fit}
\usepackage{caption}
\usepackage{amsthm}

\usepackage{amsmath}

\usepackage[pagebackref=true, bookmarksopen=true,colorlinks=true, linkcolor=red,citecolor=blue]{hyperref}

\usepackage[capitalise]{cleveref}  

\global\long\def\ii{\cap}%

\global\long\def\u{\cup}%

\global\long\def\I{\bigcap}%

\global\long\def\U{\bigcup}%

\global\long\def\s{\subset}%

\global\long\def\ud{\mathbin{\dot{\u}}}%

\global\long\def\ñ{\sim}%

\global\long\def\Le{\le}%

\global\long\def\core#1{\textnormal{core}(#1)}%

\global\long\def\corona#1{\textnormal{corona}(#1)}%

\global\long\def\a#1{\left|#1\right|}%

\usepackage{tikz-cd}

\usepackage{tkz-graph}

\usepackage{multicol}

\usepackage{subcaption}

\newtheorem{theorem}{Theorem}[section]

\newtheorem{conjecture}[theorem]{Conjecture}
\newtheorem{corollary}[theorem]{Corollary}

\newtheorem{lemma}[theorem]{Lemma}

\newtheorem{problem}[theorem]{Problem}

\newtheorem{observation}[theorem]{Observation}

\usepackage{comment}
\begin{document}

\begin{abstract}
Let $\core G$ and $\corona G$ denote the intersection and the union,
respectively, of all maximum independent sets of a graph $G$.

In this work, we show that for a graph with at most two odd cycles,
$\a{\core G}+\a{\corona G}$ is equal to $2\alpha(G)$, $2\alpha(G)+1$,
or $2\alpha(G)+2$, and we precisely characterize when each value
occurs.

We further characterize graphs with at most two odd cycles that admit the core--corona partition
$V(G)=\corona G\ud N(\core G)$, extending known results for König--Egerváry and almost bipartite graphs.

Deciding whether $\core G=\emptyset$ is known to be \textbf{NP}-hard.
As an algorithmic consequence of the obtained results, we show that
the core, independence number and the corona can  be computed in polynomial time for
this class of graphs.
	
\end{abstract}

\begin{keyword} 	Maximum Independent Set, Kőnig-Egerváry
	graph, Corona, Core, 2-bicritical, Characterization 	\MSC 05C70, 05C75 \end{keyword}
\begin{frontmatter} 
	\title{Structural and Polynomial-Time Results on Core and Corona in Odd-Bicyclic Graphs}


	
	\author[IMASL,DEPTO]{Kevin Pereyra} 	\ead{kdpereyra@unsl.edu.ar}


	\address[IMASL]{Instituto de Matem\'atica Aplicada San Luis, Universidad Nacional de San Luis and CONICET, San Luis, Argentina.}
	\address[DEPTO]{Departamento de Matem\'atica, Universidad Nacional de San Luis, San Luis, Argentina.} 	
	
	\date{Received: date / Accepted: date} 
	
\end{frontmatter} %

\section{Introduction}\label{sss0}

Let $\alpha(G)$ denote the cardinality of a maximum independent set,
and let $\mu(G)$ be the size of a maximum matching in $G=(V,E)$.
It is known that $\alpha(G)+\mu(G)$ equals the order of $G$,
in which case $G$ is a König--Egerváry graph 
\cite{deming1979independence,gavril1977testing,sterboul1979characterization}.
K\H{o}nig-Egerv\'{a}ry graphs have been extensively studied
\cite{bourjolly2009node,jarden2017two,levit2006alpha,levit2012critical,jaume5404413kr}.
It is known that every bipartite graph is a König--Egerváry graph 
\cite{egervary1931combinatorial}. These graphs were independently introduced by Deming \cite{deming1979independence}, Sterboul \cite{sterboul1979characterization}, and \cite{gavril1977testing}. A graph $G$ is almost bipartite if it has only one odd cycle.
Let $\Omega^{*}(G)=\left\{ S:S\textnormal{ is an independent set of }G\right\}$,
$\Omega(G)=\{S:S$ is a maximum independent set of $G\}$,
$\textnormal{core}(G)=\I\left\{ S:S\in\Omega(G)\right\}$ 
\cite{levit2003alpha+}, and 
$\textnormal{corona}(G)=\U\left\{ S:S\in\Omega(G)\right\}$ 
\cite{boros2002number}.

A fundamental result shows that, for every König--Egerváry graph $G$, the
equality
$|\core G|+|\corona G|=2\alpha(G)$ holds \cite{levit2011set}, while for
almost bipartite non-König--Egerváry graphs one has
$|\core G|+|\corona G|=2\alpha(G)+1$ \cite{levit2025almost}. Moreover, for
both König--Egerváry graphs \cite{levit2003alpha+} and almost bipartite
non-König--Egerváry graphs \cite{levit2025almost}, the vertex set admits
the partition
\[
V(G)=\corona G\;\dot{\cup}\;N(\core G).
\]
This partition provides strong and useful structural information for the
study of these graphs.

The purpose of this paper is to further extend this theory to graphs with
\emph{at most two odd cycles}. This class naturally extends the known
theory for König--Egerváry and almost bipartite graphs, while still
retaining enough structure to allow a precise analysis. At the same
time, it represents a boundary case where new behavior appears: graphs
in this family may exhibit several distinct values of
$|\core G|+|\corona G|$, depending on the interaction and relative
position of their odd cycles.

Our approach is based on Larson’s independence decomposition, which
splits an arbitrary graph into a König--Egerváry part and a complementary
subgraph that is $2$-bicritical \cite{larson2011critical}. This reduction
allows us to transfer known results from the König--Egerváry setting to a
non--König--Egerváry framework and to isolate the contribution of the odd
cycles. Using this method, we obtain a complete structural
characterization of the possible values of $|\core G|+|\corona G|$ for
graphs with at most two odd cycles.

In addition, we show that every graph with at most two odd cycles admits
the core--corona partition
\[
V(G)=\corona G\;\dot{\cup}\;N(\core G).
\]

From an algorithmic point of view, the core plays a particularly
important role. While deciding whether $\core G=\emptyset$ is known to be
\textbf{NP}-hard in general \cite{boros2002number}, we show that, for
graphs with at most two odd cycles, the situation is radically
different. As a consequence of our structural results, we prove that the
core and the corona can both be computed in polynomial time for this
class of graphs.

\medskip

The paper is organized as follows.
In \cref{sss0} we present the general context of the problem and
introduce the fundamental concepts.
In \cref{sss1} we fix the notation used throughout the paper.
In \cref{sss2} we devote to the structural analysis of the quantity
$|\core G|+|\corona G|$, culminating in a complete classification for
graphs with at most two odd cycles.
In \cref{sss3} we study the core--corona partition and prove that it
holds throughout this class.
Finally, \cref{sss4} discusses algorithmic consequences and establishes
polynomial-time procedures for computing $\alpha(G)$, $\core G$, and
$\corona G$, and proposes some conjectures and open problems.

\section{Preliminaries}\label{sss1}
All graphs considered in this paper are finite, undirected, and simple. 
For any undefined terminology or notation, we refer the reader to 
Lovász and Plummer \cite{LP} or Diestel \cite{Distel}.

Let \( G = (V, E) \) be a simple graph, where \( V = V(G) \) is the finite set of vertices and \( E = E(G) \) is the set of edges, with \( E \subseteq \{\{u, v\} : u, v \in V, u \neq v\} \). We denote the edge \( e=\{u, v\} \) as \( uv \). A subgraph of \( G \) is a graph \( H \) such that \( V(H) \subseteq V(G) \) and \( E(H) \subseteq E(G) \). A subgraph \( H \) of \( G \) is called a \textit{spanning} subgraph if \( V(H) = V(G) \). 

Let \( e \in E(G) \) and \( v \in V(G) \). We define \( G - e := (V, E - \{e\}) \) and \( G - v := (V - \{v\}, \{uw \in E : u,w \neq v\}) \). If \( X \subseteq V(G) \), the \textit{induced} subgraph of \( G \) by \( X \) is the subgraph \( G[X]=(X,F) \), where \( F:=\{uv \!\in\! E(G) : u, v \!\in \! X\} \).

The number of vertices in a graph $G$ is called the \textit{order} of the graph and denoted by $\left|G\right|$ or $n(G)$.
A \textit{cycle} in $G$ is called \textit{odd} (resp. \textit{even}) if it has an odd (resp. even) number of edges.

For a vertex $v\in V(G)$, the \emph{neighborhood} of $v$ is
\[
N_G(v)=\{u\in V(G): uv\in E(G)\}.
\]
When no confusion arises, we write $N(v)$ instead of $N_G(v)$. For a set $S\subseteq V(G)$, the \emph{neighborhood} of $S$ is
\[
N_G(S)=\bigcup_{v\in S} N_G(v).
\]

A \textit{matching} \(M\) in a graph \(G\) is a set of pairwise non-adjacent edges. The \textit{matching number} of \(G\), denoted by  \(\mu(G)\), is the maximum cardinality of any matching in \(G\). Matchings induce an involution on the vertex set of the graph: \(M:V(G)\rightarrow V(G)\), where \(M(v)=u\) if \(uv \in M\), and \(M(v)=v\) otherwise. If \(S, U \subseteq V(G)\) with \(S \cap U = \emptyset\), we say that \(M\) is a matching from \(S\) to \(U\) if \(M(S) \subseteq U\). A matching $M$ is \emph{perfect} if $M(v)\neq v$ for every vertex
of the graph.

A vertex set \( S \subseteq V \) is \textit{independent} if, for every pair of vertices \( u, v \in S \), we have \( uv \notin E \). 
The number of vertices in a maximum independent set is denoted by \( \alpha(G) \).  A \textit{bipartite} graph is a graph whose vertex set can be partitioned into two disjoint independent sets.

\section{Structural Aspects of $\a{\core G}+\a{\corona G}$}\label{sss2}

In this section we focus on the structural analysis of the quantity
$|\core G|+|\corona G|$. Our main goal is to determine all possible
values of this parameter for graphs with at most two odd cycles and to
characterize precisely when each value occurs. In \cite{larson2011critical}, Larson introduces the following decomposition theorem.

\begin{theorem}[\cite{larson2011critical}\label{larsonthm}]
	For any graph $G$, there is a unique set $L(G)\s V(G)$
	such that
	\begin{enumerate}
		\item $\alpha(G)=\alpha(G[L])+\alpha(G[V(G)-L(G)])$,
		\item $G[L(G)]$ is a König-Egerváry graph,
		\item for every non-empty independent set $I$ in $G[V(G)-L(G)]$, we have $\left|N(I)\right|>\left|I\right|,$
		and
		\item for every maximum critical independent set $J$ of $G$,  $L(G)=J\u N(J)$.
	\end{enumerate}
\end{theorem}

Throughout the remainder of the paper, $L(G)$ and $L^{c}(G)=V(G)-L(G)$ denote the sets
of \cref{larsonthm}; moreover, to simplify the notation, we define the induced graphs
\begin{align*}
	L_{G} & :=G[L(G)],\\
	L_{G}^{c} & :=G[L^{c}(G)].
\end{align*}

\medskip
\medskip

The notion of 2-bicritical graphs was introduced
in \cite{pulleyblank1979minimum}, and they can be characterized as follows.

\begin{theorem}[\cite{pulleyblank1979minimum}\label{1928u3123}]
	A graph $G$ is $2$-bicritical if and only if $\left|N(S)\right|>\left|S\right|$
	for every nonempty independent set $S\subseteq V(G)$.
\end{theorem}

The class of $2$-bicritical graphs can be regarded as the structural counterpart of König--Egerváry graphs \cite{larson2011critical}. In recent works, several new properties of $2$-bicritical graphs have been established; see, for instance, \cite{kevinSDKECHAR, kevinSDKEGE, kevinPOSYFACTOR}.

\medskip

\begin{observation}
	By \cref{larsonthm}, for every graph $G$ with $L^c(G)\neq \emptyset$, it follows that $L_{G}^{c}$ is a 2-bicritical graph.
\end{observation}


\begin{theorem}[\label{oij123ioj132ioj}\cite{kevincorecorona2biciclicbicritical}]
	Let $G$ be a $2$-bicritical graph with two odd cycles, then:
	\begin{itemize}
		\item If $G$ is connected and the two odd cycles share at most one vertex, then
		$\a{\core G}+\a{\corona G}=2\alpha(G)$,
		\item If both odd cycles share at least two vertices, then
		$\a{\core G}+\a{\corona G}=2\alpha(G)+1$,
		\item If $G$ is disconnected, then
		$\a{\core G}+\a{\corona G}=2\alpha(G)+2$.
	\end{itemize}
\end{theorem}

\begin{theorem}
	[\label{qoijhj12i3}\cite{levit2012critical}] A graph $G$ is a K\H{o}nig--Egerváry graph if and
	only if every $S\in\Omega(G)$ is critical.
\end{theorem}

\begin{theorem}
	[\label{qoijhj12i3qoijhj12i3}Butenko \& Trukhanov\cite{butenko2007using}] If $I$ is a critical
independet set in a graph $G$ then there is a maximum independent
set $S$ in $G$ such that $I\s S$. 
\end{theorem}

\begin{lemma}\label{i123ij1i2i2i}
	Let $G$ be a K\H{o}nig--Egerváry graph. Then $S$
	is a maximum independent set if and only if $S$ is a maximum
	critical independent set.
\end{lemma}
\begin{proof}
	It follows directly from \cref{qoijhj12i3} and \cref{qoijhj12i3qoijhj12i3}.
\end{proof}

\begin{theorem}
	[\cite{jarden2019monotonic}\label{1230ij}] $\a{\corona G}+\a{\core G}=2\alpha(G)$
	if and only if $\a{\U\Gamma}+\a{\I\Gamma}=2\alpha(G)$ holds for each
	non-empty $\Gamma\s\Omega(G)$. 
\end{theorem}

\begin{theorem}\label{0i1j230i123}
	If G is a König--Egerváry graph, then
	\begin{itemize}
		\item \textnormal{\cite{levit2003alpha+}} $\corona G\u N\left(\core G\right)=V(G).$
		\item \textnormal{\cite{levit2011set}} $\a{\core G} +\a{\corona G}=2\alpha(G).$
	\end{itemize}
\end{theorem}

From \cref{1230ij} and \cref{0i1j230i123}, the following follows directly.

\begin{corollary}\label{saodi123}
	If $G$ is a K\H{o}nig-Egerváry graph, then 
	\[
	\a{\U\Gamma}+\a{\I\Gamma}=2\alpha(G)
	\]
	\noindent holds for each non-empty $\Gamma\s\Omega(G)$.
\end{corollary}

The next result shows how the parameters $\core G$ and $\corona G$
behave under Larson’s decomposition, and allows us to reduce the study
of $|\core G|+|\corona G|$ to the $2$-bicritical component.

\begin{theorem}\label{89uih9uiy9iu}
	Let $G$ be a graph. Then
	\[
	\a{\core G}+\a{\corona G}
	=2\alpha(L_{G})+\a{\core{L_{G}^{c}}}+\a{\corona{L_{G}^{c}}}.
	\]
\end{theorem}

\begin{proof}
	Let $I$ be a maximum critical independent set. Then, by \cref{larsonthm},
	$L(G)=I\u N(I)$. From \cref{larsonthm}, we have
	$\alpha(G)=\alpha(L_{G})+\alpha(L_{G}^{c})$, and $L_{G}$ is a
	K\H{o}nig--Egerváry graph. Hence, by \cref{i123ij1i2i2i}, there exists
	$T\in\Omega(G)$ such that $T\ii L(V)=I$, where $\a I=\alpha(L_{G})$.
	Then $T\ii L^{c}(G)$ can be any set in $\Omega(L_{G}^{c})$, that is,
	
	\begin{align*}
		\corona G & =\U_{S\in\Omega(G)}S\\
		& =\left(L(G)\ii\U_{S\in\Omega(G)}S\right)
		\u\left(L^{c}(G)\ii\U_{S\in\Omega(G)}S\right)\\
		& =\left(L(G)\ii\U_{S\in\Omega(G)}S\right)\u\corona{L_{G}^{c}}.
	\end{align*}
	
	\noindent Similarly, we show that
	\begin{align*}
		\core G & =\left(L(G)\ii\I_{S\in\Omega(G)}S\right)\u\core{L_{G}^{c}}.
	\end{align*}
	
	\noindent Therefore,
	
	\begin{eqnarray*}
		\a{\core G}+\a{\corona G} & = &
		\a{L(G)\ii\U_{S\in\Omega(G)}S}
		+\a{L(G)\ii\I_{S\in\Omega(G)}S}\\
		&  & +\a{\corona{L_{G}^{c}}}+\a{\core{L_{G}^{c}}}.
	\end{eqnarray*}
	
	\noindent But by \cref{larsonthm}, $L_{G}$ is a K\H{o}nig--Egerváry graph.
	Therefore, by \cref{saodi123},
	\[
	\a{L(G)\ii\U_{S\in\Omega(G)}S}
	+\a{L(G)\ii\I_{S\in\Omega(G)}S}
	=2\alpha(L_{G}).
	\]
	
	\noindent Thus, the desired equality follows.
\end{proof}

\begin{corollary}\label{asoij1i2j03}
	Let $G$ be a graph and let $k$ be a nonnegative integer
	such that $\a{\core{L_{G}^{c}}}+\a{\corona{L_{G}^{c}}}=2\alpha(L_{G}^{c})+k$.
	Then $\a{\core G}+\a{\corona G}=2\alpha(G)+k$.
\end{corollary}

\begin{proof}
	Note that by \cref{89uih9uiy9iu}
	\begin{align*}
		\a{\core G}+\a{\corona G} & =2\alpha(L_{G})+\a{\core{L_{G}^{c}}}+\a{\corona{L_{G}^{c}}}\\
		& =2\alpha(L_{G})+2\alpha(L_{G}^{c})+k.
	\end{align*}
	Then, by the additivity of $\alpha$ in Larson’s decomposition,
	the above reduces to $\alpha(G)+k$, see \cref{larsonthm}.
\end{proof}

Thus, from \cref{oij123ioj132ioj} and \cref{asoij1i2j03}, we obtain a complete classification of
the possible values of $|\core G|+|\corona G|$ for graphs with at most
two odd cycles.

\begin{theorem}
	Let $G$ be a graph such that $L_{G}^{c}$ has at most two odd cycles.
	Then
	\begin{itemize}
		\item If $L^{c}(G)=\emptyset$, or if $L_{G}^{c}$ is connected and the two odd cycles share at most one vertex, then
		$\a{\core G}+\a{\corona G}=2\alpha(G)$,
		\item If $L_{G}^{c}$ has a unique odd cycle, or  both odd cycles share at least two vertices, then
		$\a{\core G}+\a{\corona G}=2\alpha(G)+1$,
		\item If $L_{G}^{c}$ has two vertex-disjoint odd cycles and is disconnected,
		then $\a{\core G}+\a{\corona G}=2\alpha(G)+2$.
	\end{itemize}
\end{theorem}

\begin{corollary}
	Let $G$ be a graph with at most two odd cycles.
	Then
	\[
	2\alpha(G)\Le\a{\core G}+\a{\corona G}\le2\alpha(G)+2.
	\]
\end{corollary}

\section{The Core–Corona Partition}\label{sss3}

Note that, for every graph $G$, the sets $\corona G$ and $N(\core G)$
are disjoint. Indeed, if $x \in N(\core G$, then $x$ is adjacent
to some vertex that belongs to every maximum independent set and,
consequently, $x$ cannot belong to any maximum independent set.

It is therefore of particular interest to study those graphs $G$ such that
\[
\corona G\cup N(\core G) = V(G),
\]
that is, those for which these sets induce a partition of $V(G)$.
This property imposes a strong structure on the graph.

It is known that this equality holds for Kőnig--Egerváry graphs
\cite{levit2003alpha+} and for almost-bipartite non--Kőnig--Egerváry graphs
\cite{levit2025almost}. In \cref{asoikjdo123asoikjdo123}, we provide a reductive
characterization of the class of graphs satisfying this property.
Moreover, in \cref{oi123ji12i}, we prove that every graph with at most two odd
cycles admits this partition of $V(G)$.

\begin{lemma}[\cite{kevincoronacore}\label{asdouih12}]
	Let $G$ be a K\H{o}nig--Egerváry graph and let $\emptyset\neq\Gamma\subseteq\Omega(G)$.
	Then there exists a matching from $V(G)-\U\Gamma$ into $\I\Gamma$.
\end{lemma}

In \cite{jarden2019monotonic} it is shown that \cref{asdouih12} holds when
$\Gamma=\Omega(G)$. From \cref{asdouih12} the following result follows
immediately.

\begin{theorem}[\cite{kevincoronacore}\label{asdouih129i3}]
	Let $G$ be a K\H{o}nig--Egerváry graph and let $\emptyset\neq\Gamma\subseteq\Omega(G)$.
	Then
	\[
	\U\Gamma\cup N\left(\I\Gamma\right)=V(G).
	\]
\end{theorem}

\begin{lemma}\label{lasldllas}
	For every graph $G$, it holds that $\core G\ii L^{c}(G)=\core{L_{G}^{c}}$.
	Moreover, $\corona G\ii L^{c}(G)=\corona{L_{G}^{c}}$.
\end{lemma}

\begin{proof}
	Recalling the equalities obtained in the proof of \cref{89uih9uiy9iu},
	\begin{align*}
		\corona G & =\left(L(G)\ii\U_{S\in\Omega(G)}S\right)\u\corona{L_{G}^{c}},\\
		\core G & =\left(L(G)\ii\I_{S\in\Omega(G)}S\right)\u\core{L_{G}^{c}}.
	\end{align*}
	Taking the intersection with $L^{c}(G)$ in both equalities above,
	the result follows immediately.
\end{proof}

\begin{theorem}\label{asoikjdo123asoikjdo123}
	For every graph $G$, it holds that $\corona G\ud N_G(\core G)=V(G)$
	if and only if $\corona{L_{G}}\ud N_{L^c_G}(\core{L_{G}})=L^{c}(G)$.
\end{theorem}

\begin{proof}
	As in the proof of \cref{89uih9uiy9iu}, we have
	\begin{align*}
		\corona G & =\left(L(G)\ii\U_{S\in\Omega(G)}S\right)\u\corona{L_{G}^{c}},\\
		\core G & =\left(L(G)\ii\I_{S\in\Omega(G)}S\right)\u\core{L_{G}^{c}}.
	\end{align*}
	But by \cref{larsonthm}, $L_{G}$ is a K\H{o}nig--Egerváry graph; therefore,
	by \cref{asdouih129i3},
	\[
	L(G)\s\left(L(G)\ii\U_{S\in\Omega(G)}S\right)\u
	N\left(L(G)\ii\I_{S\in\Omega(G)}S\right).
	\]
	
	On the other hand, suppose that
	$\corona{L_{G}}\ud N(\core{L_{G}})=L^{c}(G)$. Then we have
	
	\begin{eqnarray*}
		\corona G\u N(\core G) & = &
		\left(L(G)\ii\U_{S\in\Omega(G)}S\right)\u\corona{L_{G}^{c}}\\
		&  & \u N\left(\left(L(G)\ii\I_{S\in\Omega(G)}S\right)\u\core{L_{G}^{c}}\right)\\
		& \supset & \corona{L_{G}^{c}}\u N(\core{L_{G}^{c}})\\
		&  & \left(L(G)\ii\U_{S\in\Omega(G)}S\right)\u
		N\left(L(G)\ii\I_{S\in\Omega(G)}S\right)\\
		& \supset & L^{c}(G)\u L(G)\\
		& = & V(G).
	\end{eqnarray*}
	
	Conversely, suppose that $\corona G\ud N(\core G)=V(G)$.
	Let $I$ be a maximum critical independent set; then, by \cref{larsonthm},
	$L(G)=I\u N(I)$. As in the proof of \cref{89uih9uiy9iu}, there exists
	$T\in\Omega(G)$ such that $T\ii L(G)=I$, where $\a I=\alpha(L_{G})$.
	Thus,
	\[
	\core G\ii L(G)\s I.
	\]
	
	Let $v\in L^{c}(G)$. If $v\in N(\core G)$, then by \cref{lasldllas},
	\[
	v\in N(\core G\ii L^{c}(G))=N(\core{L_{G}^{c}}).
	\]
	
	Otherwise, suppose that $v\in\corona G$. Then, by \cref{lasldllas},
	$v\in\corona{L_{G}^{c}}$. Therefore, we have proved that
	$L^{c}(G)\s\corona{L_{G}^{c}}\u N(\core{L_{G}^{c}})$.
\end{proof}

\begin{theorem}[\label{asoikjdo123}\cite{kevincorecorona2biciclicbicritical}]
	Let $G$ be a $2$-bicritical graph with at most two odd cycles.
	Then $\corona G\ud N(\core G)=V(G)$ if and only if $G$ does not have two odd cycles sharing exactly one vertex.
\end{theorem}

\begin{theorem}\label{oi123ji12i}
	Let $G$ be a graph such that $L_{G}^{c}$ has at most two odd cycles.
	Then $\corona G\ud N(\core G)=V(G)$ if and only if $G$ does not have two odd cycles sharing exactly one vertex.
\end{theorem}

\begin{proof}
	It follows directly from \cref{asoikjdo123} and \cref{asoikjdo123asoikjdo123}.
\end{proof}

\section{Algorithmic Consequences}\label{sss4}

In this final section we discuss algorithmic consequences of the
structural results obtained in the previous sections.

\begin{problem}\label{asio123ji12as}
	Let $G$ be a 2-bicritical graph with at most two odd cycles.
	Determine the independence number $\alpha (G)$.
\end{problem}

\begin{problem}\label{asio123ji12asXXX}
	Let $G$ be a graph with at most two odd cycles.
	Determine the independence number $\alpha (G)$.
\end{problem}

\begin{theorem}[\cite{kevincorecorona2biciclicbicritical}\label{aij3i2j1i}]
	Let $G$ be a $2$-bicritical graph with at most two odd cycles.
	If $G$ is connected, then
	\[
	\alpha(G)+\mu(G)=n-1.
	\]
	If $G$ is disconnected, then
	\[
	\alpha(G)+\mu(G)=n-2.
	\]
\end{theorem}

It is well known that the matching number $\mu(G)$ can be computed in
polynomial time \cite{edmonds1965paths}. Therefore, as an immediate consequence of
\cref{aij3i2j1i}, we obtain the following result.

\begin{lemma}
	The \cref{asio123ji12as} are solvable in polynomial time.
\end{lemma}

\begin{problem}\label{asio123ji12}
	Let $G$ be a graph with at most two odd cycles.
	Determine the core, $\core G$.
\end{problem}

\begin{problem}\label{asio123ji12asio123ji12}
	Let $G$ be a graph with at most two odd cycles.
	Determine the corona, $\corona G$.
\end{problem}

\begin{theorem}[\cite{levit2003alpha+}\label{TheoremXXXLM}]
	If $G$ is a K\H{o}nig-Egerváry graph, then $N(\core G)=V(G)-\corona G.$
\end{theorem}

\begin{theorem}\label{asij123j}
	\cref{asio123ji12} and \ref{asio123ji12asio123ji12} are solvable in polynomial time.
\end{theorem}

\begin{proof}
	Finding a maximum critical independent set of a graph can be done
	in polynomial time \cite{zhang1990finding}. Consequently, the Larson
	decomposition $L(G),L^{c}(G)$ can be computed in polynomial time;
	see \cref{larsonthm}.
	
	Let $G$ be a graph with at most two odd cycles. By \cref{larsonthm}, if $L^{c}(G)\neq\emptyset$, then
	the subgraph $L^{c}_{G}$ is a $2$-bicritical graph with at most two odd
	cycles. Hence, by \cref{aij3i2j1i}, the value $\alpha(L^{c}_{G})$ can be
	computed in polynomial time. On the other hand, $L_{G}$ is a
	K\H{o}nig--Egerv\'ary graph, that is,
	$\alpha(L_{G})+\mu(L_{G})=|L(G)|$, and therefore $\alpha(L_{G})$ can
	also be computed in polynomial time. By \cref{larsonthm},
	\[
	\alpha(G)=\alpha(L_{G})+\alpha(L^{c}_{G}),
	\]
	and thus $\alpha(G)$ can be computed in polynomial time. If $L^{c}(G)=\emptyset$, then $G=L_{G}$ is a
	K\H{o}nig--Egerv\'ary graph, and therefore $\alpha(G)$ can be computed
	in polynomial time as well.
	
	For every vertex $v\in V(G)$, the graph $G-v$ has at most two odd
	cycles. Hence, $\alpha(G-v)$ can be computed in polynomial time.
	Therefore, the equality $\alpha(G-v)=\alpha(G)$ can be checked in
	polynomial time. By definition, $v\in\core G$ if and only if
	$\alpha(G-v)\neq\alpha(G)$. Consequently, \cref{asio123ji12}
	is solvable in polynomial time.

	On the other hand, since we know $L_{G}^{c}$, we can by BFS/DFS find
	an odd cycle $C$, deleting one vertex at a time from $C$ in $G$
	and verifying whether the obtained graph is bipartite, and in this way we can find
	the intersection of both odd cycles in polynomial time. 
	Suppose that $L^c_G$ does not have two odd cycles that share exactly one vertex, then by \cref{oi123ji12i}, the vertex set of $G$ admits
	the partition
	\[
	V(G)=\corona G \,\dot{\cup}\, N(\core G).
	\]
	Since $N(\core G)$ can be computed in polynomial time once $\core G$
	is known, it follows that
	\[
	\corona G = V(G)\setminus N(\core G)
	\]
	can also be determined in polynomial time.  If the intersection of both odd cycles is a single vertex $x$,
	then it is easy to see that $x\notin\corona G$, moreover $G-x$ is a
	bipartite graph with $\alpha(G-x)=\alpha(G)$ and $\corona G=\corona{G-x}$ and $\core{G}=\core{G-x}$.
	But analogously by \cref{TheoremXXXLM} applied to $G-x$ it follows that 
	\[
	\corona G = V(G)\setminus N(\core G)
	\]
	Therefore,
	\cref{asio123ji12asio123ji12} is solvable in polynomial time. 
\end{proof}

Using the same approach as in the proof of \cref{asij123j}, namely by
computing the Larson decomposition and applying
\cref{aij3i2j1i} together with the additivity of $\alpha$, we obtain the
following result.

\begin{theorem}
	The \cref{asio123ji12asXXX} are solvable in polynomial time.
\end{theorem}

To conclude, we point out a natural direction for further research.

\begin{problem}
	Characterize all graphs that satisfy the core--corona partition
	\[
	V(G)=\corona G \,\dot{\cup}\, N(\core G).
	\]
\end{problem}

\begin{conjecture}
	Let $G$ be a graph such that every pair of odd cycles in $L_{G}^{c}$
	is vertex-disjoint. Then $\core G$ and $\corona G$ can be computed in
	polynomial time.
\end{conjecture}

\section*{Acknowledgments}

	This work was partially supported by Universidad Nacional de San Luis, grants PROICO 03-0723 and PROIPRO 03-2923, MATH AmSud, grant 22-MATH-02, Consejo Nacional de Investigaciones
	Cient\'ificas y T\'ecnicas grant PIP 11220220100068CO and Agencia I+D+I grants PICT 2020-00549 and PICT 2020-04064.

	\section*{Declaration of generative AI and AI-assisted technologies in the writing process}
	During the preparation of this work the authors used ChatGPT-3.5 in order to improve the grammar of several paragraphs of the text. After using this service, the authors reviewed and edited the content as needed and take full responsibility for the content of the publication.

\section*{Data availability}

Data sharing not applicable to this article as no datasets were generated or analyzed during the current study.

\section*{Declarations}

\noindent\textbf{Conflict of interest} \ The authors declare that they have no conflict of interest.

\bibliographystyle{apalike}

\bibliography{TAGcitasV2025}

\end{document}